\documentclass[runningheads]{llncs}
\usepackage{graphicx} 

\usepackage[utf8]{inputenc}

\usepackage[T1]{fontenc}
\usepackage{ucs}
\usepackage{amsmath}
\usepackage{amsfonts}
\usepackage{amssymb}
\usepackage{xcolor}

\usepackage{mathtools}
\usepackage{pdfpages}

\usepackage{listings}

\usepackage{hyperref}

\usepackage{booktabs}
\usepackage{algorithm}
\usepackage{algorithmic}
\usepackage{enumitem}
\usepackage{ stmaryrd }

\usepackage{listings}
\usepackage[appendix=inline]{apxproof}

\newtheoremrep{theorema}[theorem]{Theorem}
\newtheoremrep{lemmaa}[lemma]{Lemma}

\usepackage{tikz}
\usepackage{standalone} % Necessary to skip headers of {standalone} documents
\usepackage{tikzscale} % To use \includegraphics with TikZ code
\usetikzlibrary{backgrounds,fit,positioning}
\usetikzlibrary{decorations.pathmorphing}
\usetikzlibrary {automata,positioning}
\tikzset{snake it/.style={decorate, decoration=snake}}

\newcommand{\val}[1]{\|#1\|}

\DeclarePairedDelimiter\ceil{\lceil}{\rceil}
\DeclarePairedDelimiter\floor{\lfloor}{\rfloor}

\newcommand{\la}{\langle}
\newcommand{\ra}{\rangle}

\newcommand{\longversion}[1]{}
\renewcommand{\longversion}[1]{#1}

\makeatletter
\DeclareFontFamily{OMX}{MnSymbolE}{}
\DeclareSymbolFont{MnLargeSymbols}{OMX}{MnSymbolE}{m}{n}
\SetSymbolFont{MnLargeSymbols}{bold}{OMX}{MnSymbolE}{b}{n}
\DeclareFontShape{OMX}{MnSymbolE}{m}{n}{
	<-6>  MnSymbolE5
	<6-7>  MnSymbolE6
	<7-8>  MnSymbolE7
	<8-9>  MnSymbolE8
	<9-10> MnSymbolE9
	<10-12> MnSymbolE10
	<12->   MnSymbolE12
}{}
\DeclareFontShape{OMX}{MnSymbolE}{b}{n}{
	<-6>  MnSymbolE-Bold5
	<6-7>  MnSymbolE-Bold6
	<7-8>  MnSymbolE-Bold7
	<8-9>  MnSymbolE-Bold8
	<9-10> MnSymbolE-Bold9
	<10-12> MnSymbolE-Bold10
	<12->   MnSymbolE-Bold12
}{}
\DeclareMathDelimiter{\ulcorner}
{\mathopen}{MnLargeSymbols}{'036}{MnLargeSymbols}{'036}
\DeclareMathDelimiter{\urcorner}
{\mathclose}{MnLargeSymbols}{'043}{MnLargeSymbols}{'043}
\makeatother

\DeclareFontFamily{U}{MnSymbolC}{}
\DeclareSymbolFont{MnSyC}{U}{MnSymbolC}{m}{n}
\DeclareMathSymbol{\boxdot}{\mathbin}{MnSyC}{"76}
\DeclareMathSymbol{\diamonddot}{\mathbin}{MnSyC}{"7E}
\DeclareFontShape{U}{MnSymbolC}{m}{n}{
    <-6>  MnSymbolC5
   <6-7>  MnSymbolC6
   <7-8>  MnSymbolC7
   <8-9>  MnSymbolC8
   <9-10> MnSymbolC9
  <10-12> MnSymbolC10
  <12->   MnSymbolC12}{}

\DeclareSymbolFont{extraup}{U}{zavm}{m}{n}
\DeclareMathSymbol{\vardiamond}{\mathalpha}{extraup}{87}

\makeatletter
\newcommand{\dotminus}{\mathbin{\text{\@dotminus}}}

\newcommand{\@dotminus}{%
	\ooalign{\hidewidth\raise1ex\hbox{.}\hidewidth\cr$\m@th-$\cr}%
}
\makeatother

\DeclareRobustCommand{\ontop}{\genfrac{}{}{0pt}{}}
\begin{document}
\title{The universal tangle for spatial reasoning\thanks{
Supported by the FWO-FWF Lead Agency grant G030620N (FWO)/I4513N (FWF) and by the SNSF--FWO Lead Agency Grant 200021L\_196176/G0E2121N.}}
\author{David Fern\'andez-Duque\inst{1,2}\orcidID{0000-0001-8604-4183} \and
Konstantinos Papafilippou\inst{1}\orcidID{0000-0002-2831-0575} }
\authorrunning{D. Fern\'andez-Duque \& K. Papafilippou}
\institute{Ghent University, Ghent, Belgium\\
\email{Konstantinos.Papafilippou@UGent.be} \and
University of Barcelona, Barcelona, Spain\\
\email{fernandez-duque@ub.edu}}
\maketitle              % typeset the header of the contribution
\begin{abstract}
The topological $\mu$-calculus has gathered attention in recent years as a powerful framework for representation of spatial knowledge.
In particular, spatial relations can be represented over finite structures in the guise of weakly transitive ({\sf wK4}) frames.
In this paper we show that the topological $\mu$-calculus is equivalent to a simple fragment based on a variant of the `tangle' operator.
Similar results were proven for transitive frames by Dawar and Otto, using modal characterisation theorems for the corresponding classes of frames. However, since these theorems are not available in our setting, which has the upshot of providing a more explicit translation and upper bounds on formula size.
\end{abstract}
\section{Introduction}

Qualitative spatial reasoning aims to capture basic relations between regions in space in a way that is computationally efficient and thus suitable for knowledge representation and AI (see \cite{Cohn,Stell} for overviews).
The {\em region connection calculus} ($\sf RCC8$)~\cite{RCC8,RCC8b} deals with relations such as `partially overlaps' (e.g.~Mexico and Mesoamerica) or `is a non-tangential proper part' (e.g.~Paraguay and South America) while avoiding undecidability phenomena by not allowing for quantification over points or regions.

$\sf RCC8$ can be embedded into modal logic ($\sf ML$) with a universal modality~\cite{WolterZ00}.
This allows us to import many techniques from $\sf ML$, including the representation of regions using transitive Kripke frames, i.e.~pairs $\langle W,\sqsubset\rangle$, where $W$ is a set of points and $\sqsubset$ is a transitive relation representing `nearness'.
It also tells us that little is lost by omitting quantifiers, due to so-called modal characterization theorems~\cite{JvBThesis}, which state that $\sf ML$ is the bisimulation-invariant fragment of first order logic ($\sf FOL$), while its extension to the modal $\mu$-calculus is the bisimulation-invariant fragment of monadic second order logic ($\sf MSO$) \cite{JW}.

However, these results apply to frames where $\sqsubset$ is an arbitrary relation, whereas Dawar and Otto~\cite{DO} showed that the situation over finite, transitive frames is subtle.
In this setting, the bisimulation-invariant fragments of  $\sf FOL$ and $\sf MSO$ coincide, but are stronger than modal logic.
They are in fact equal to the $\mu$-calculus, but this in turn can be greatly simplified to its {\em tangled} fragment, which adds expressions of the form $\diamond^\infty\{\varphi_1,\ldots,\varphi_n\}$, stating that there is an accessible cluster of reflexive points where each $\varphi_i$ is satisfied.

Finite, transitive frames are suitable for representing spatial relations on metric spaces, such as Euclidean spaces or the rational numbers~\cite{Lucero-Bryan13,GH17}.
However, for the more general setting of topological spaces, one must consider a wider class of frames called {\em weakly transitive} frames: a relation $\sqsubset$ is weakly transitive if $x\sqsubset y\sqsubset z$ implies $x\sqsubseteq z$.
The modal logic of finite, weakly transitive frames is precisely that of all topological spaces~\cite{esakia-derivative}, and this result extends to the full $\mu$-calculus~\cite{BBFMu}.
In this spatial setting, Dawar and Otto's tangled operator becomes the {\em tangled derivative}, the largest subspace in which two or more sets are dense: for example, the tangle of $\mathbb Q$ and $\mathbb R\setminus\mathbb Q$ is the full real line, since the rationals and the irrationals are both dense in $\mathbb R$.
In the case of a single subset $A$, $\Diamond^\infty\{A\}$ is the {\em perfect core} of $A$, i.e.~its largest perfect subset, a notion useful in describing the limit of learnability after iterated measurements~\cite{SurpriseKR}.

Alas, over the class of weakly transitive frames, the tangled derivative is not as expressive as the $\mu$-calculus \cite{BBFMu}, which is in turn less expressive than the bisumulation-invariant fragment of $\sf MSO$, so Dawar and Otto's result fails.
Gougeon~\cite{Quent21}   proposed a more expressive operator, which here we simply dub the {\em tangle} and denote by $\vardiamond^\infty$, which coincides with the tangled derivative over metric spaces (and other spaces satisfying a regularity property known as $T_D$ spaces), but is strictly more expressive over the class of   topological spaces.
While this tangle cannot be as expressive as the bisimulation-invariant fragment of $\sf MSO$, it was still conjectured to be as expressive as the   $\mu$-calculus, thus providing a streamlined framework for representing spatial properties relevant for the learnability framework of \cite{SurpriseKR}.
This conjecture is supported by the recent result stating that the topological $\mu$-calculus collapses to its alternation-free fragment~\cite{alternation}.

In this paper we give an affirmative answer to this conjecture.
Moreover, since we cannot use games for $\sf FOL$ to establish our results, our proof uses new methods which have the advantage of providing an explicit translation of the $\mu$-calculus into tangle logic.
Among other things, we provide an upper bound on formula size, which is doubly exponential.
It is not clear if this can be greatly improved, given the exponential lower bounds of~\cite{Iliev}.

Despite the spatial motivation for the $\mu$-calculus over {\sf wK4}, the results of \cite{BBFMu} allow us to work   within the class of weakly transitive frames; since their logic is that of all topological spaces, our expressivity results lift to that context as well.
The upshot is that background in topology is not needed to follow the text.

\subsection*{Layout}
In Section \ref{sect: prelim} we review the $\mu$-calculus, present Gougeon's tangle and some basic semantic notions over path-finite weakly transitive ({\sf wK4}) frames. Section \ref{sect: finality} begins with a review of finality as used in~\cite{BBFMu}, as well as establishing additional properties we need. In Section \ref{sect: structural evaluation} we construct some formulae in the tangle logic that peer into the structure of a given Kripke model, which we use to show that the $\mu$-calculus is equivalent to the tangle logic and strictly weaker than the bisimulation invariant part of first order logic over finite and path finite {\sf wK4} frames.

\section{Preliminaries}\label{sect: prelim}

As is often the case when working with $\mu$-calculi, it will be convenient to define the $\mu$-calculus with each of the positive operations, including $\nu x. \varphi$, as primitive, and with negation being only subsequently defined.  
\begin{definition}
    The language of the modal $\mu$-calculus $\mathcal{L}_{\mu}$ is defined by the following syntax:
   \begin{align*}
       \varphi ::= \top \, | \, x \, | \, p \, | \, \neg p\, |\, \varphi \wedge \varphi \, |\, \varphi \vee \varphi \, |\, \Diamond \varphi \,|\, 
    \Box \varphi \,|\, \nu x. \varphi(x) \,|\, \mu x. \varphi (x) 
   \end{align*}
    where $x$ belongs to a set of `variables' and $p$ to a set of `constants', denoted $\mathbb{P}$.
    
    Under this presentation of the language, the formulas are said to be in \emph{negation normal form}. 
    Negation is defined classically as usual with 
     $\neg \nu x. \varphi(x) := \mu x. \neg \varphi (\neg x) $ and $\neg \mu x. \varphi(x) := \nu x. \neg \varphi (\neg x) $.
    We also write $\diamonddot \varphi := \varphi \vee \Diamond \varphi$ and similarly $\boxdot \varphi:= \varphi \wedge \Box \varphi$.
\end{definition}

The following is the standard semantics for the $\mu$-calculus over frames with a single relation $\sqsubset$ (or $\sqsubset_M$, to specify the frame).

\begin{definition}\label{defSem}
A Kripke frame is a tuple $\mathcal{F}= \langle M,\sqsubset_M\rangle$ where ${\sqsubset_M} \subseteq M\times M$.
A Kripke model is a triple $\mathcal{M}= \langle M,\sqsubset_M,\| \cdot \|_M\rangle$ where $\langle M,\sqsubset_M \rangle$ is a Kripke frame with a valuation $\| \cdot \|_M:\mathbb{P}\to  \mathcal{P}(M)$. In the sequel, we will use $\mathcal{M}$ and $M$ interchangeably. We denote the reflexive closure of $\sqsubset_M$ by $\sqsubseteq_M$.  

Given $A \subseteq M$, we denote the irreflexive and reflexive upsets of $A$ as $A{\uparrow_M} := \{ w \in M : \exists v \in A \ v \sqsubset_M w\} $ and $A{\uparrow^*_M} := A{\uparrow_M} \cup A$ respectively. The downsets are similarly denoted as $A{\downarrow_M}:= \{ w \in M : \exists v \in A \ w \sqsubset_M v\} $ and $A{\downarrow^*_M}:= A{\downarrow_M} \cup A$ respectively. We will omit the $M$ in the subscript when we will be only referring to a single model.

The valuation $\| \cdot \| = \| \cdot \|_M$ is defined as usual on Booleans with:
\medskip

\noindent \begin{tabular}{lcl}
$\| \Diamond \varphi \| := \| \varphi \| {\downarrow} $
& \ \ \ \ \ &
$\|\mu x. \varphi (x)\|:= \bigcap \{ X \subseteq M : X = \|\varphi(X)\| \} $
\\
$\| \Box \varphi \| := M\setminus (( M\setminus \| \varphi \|) {\downarrow}) $
&&
$\|\nu x. \varphi (x)\|:= \bigcup \{ X \subseteq M : X = \|\varphi(X)\| \} $
\end{tabular}
\medskip

\noindent Given a Kripke model $M$ and a world $w \in M$ we say a formula $\varphi$ is satisfied by $M$ at the world $w$ and write $w \vDash_M \varphi$ iff $w \in \| \varphi \|_M$.

A formula $\varphi$ is {\em valid} over a class of models $\Omega $ if for every $M\in\Omega$, $\|\varphi\|_M=M$.
\end{definition}

We note that $\mu x. \varphi (x)$ and $\nu x. \varphi (x)$ are the least and greatest fixed points, respectively, of the operator $X\mapsto \varphi(X)$.

We will mostly concern ourselves only with weakly transitive frames.
A relation $R$ is weakly transitive iff for all $a,b,c$ where $a \neq c$, if $aRb $ and $ bRc$ then $aRc$. A frame or model is {\em weakly transitive} if its accessibility relation is.

\begin{example}\label{exmodel}
Consider a   frame $\mathcal F$ consisting of two irreflexive points $\{0,1\}$ such that $0\sqsubset 1$ and $1\sqsubset 0$; this frame is weakly transitive since $x\sqsubset y\sqsubset z$ implies $x=z$, but it is not transitive since e.g.~$0\sqsubset 1\sqsubset 0$ but $0\not\sqsubset 0$.
To extend this frame into a model, we assign subsets of $\{0,1\}$ to each propositional variable. Assume that our variables are $e$ (even), $o$ (odd), $p$ (positive) and $i$ (integer).
We   obtain a valuation $\|\cdot\|$ if we let $\|e\| = \{0\}$, $\|o\| = \{1\}$, $\|p\| = \{1\}$, and $\|i\| = \{0,1\}$.
Then, $\|o\vee \Diamond p\| =\{0,1\}$, since every element of our model is either odd or has an accessible positive point.
We may say that this formula is {\em valid} in our model.
\end{example}
 
Recall that a topological space is a pair $\langle X,\mathcal T\rangle$, where $\mathcal T$ is a family of subsets of $X$ (called the {\em open sets}) closed under finite intersections and arbitrary unions.
If $A\subseteq X$, $d(A)$ is the set of points $x\in X$ such that whenever $x\in U$ and $U$ is open, there is $y\in A\cap U\setminus \{x\}$; this is the set of {\em limit points} of $A$.
The topological semantics for the $\mu$-calculus is obtained by modifying Definition~\ref{defSem} by setting $\|\Diamond\varphi\|=d\|\varphi\|$.
This is the basis to the modal approach to spatial reasoning, but the following allows us to work with weakly transitive frames instead.

\begin{theorem}[\cite{BBFMu}]
For $\varphi\in\mathcal L_\mu$, the following are equivalent:
\begin{itemize}

\item $\varphi$ is valid over the class of all topological spaces.

\item $\varphi$ is valid over the class of all weakly transitive frames.

\item $\varphi$ is valid over the class of all finite, irreflexive, weakly transitive frames.

\end{itemize}
\end{theorem}

This extends results of Esakia for the purely modal setting~\cite{esakia-derivative}.
Next we recall bisimulations (see e.g.~\cite{Chagrov1997ModalL}), which are binary relations preserving truth of $\mu$-calculus formulas that will be very useful in the rest of the text.

\begin{definition}
    Given $P\subseteq \mathbb{P}$ a $P$-bisimulation is a relation $\iota \subseteq M \times N$ such that, whenever $\la u,v \ra \in \iota$:
    \begin{description}
        \item[atoms] $w\vDash_M p \, \Leftrightarrow v \vDash_M p$ for all $p \in P$;
        \item[forth] If $u \sqsubset_M u'$, then there is $v \sqsubset_N v'$ such that $\la u' , v' \ra \in \iota$;
        \item[back] If $v \sqsubset_N v'$, then there is $u \sqsubset_M u'$ such that $\la u' , v' \ra \in \iota$;
        \item[global] $dom(\iota )= M$ and $rng(\iota) = N$.
    \end{description}
    Two models are called $P$-bisimilar and we write $M \rightleftharpoons_P N$ if there is some $P$-bisimulation relation between them. Given subsets $A\subseteq M$ and $B \subseteq N$, we write $A \rightleftharpoons_P B$ when $M\upharpoonright A \rightleftharpoons_P  N \upharpoonright B$, where $\upharpoonright$ denotes the usual restriction to a subset of the domain.
%    \begin{align*}
 %       \la A, {\sqsubset_M} {\upharpoonright}_{A^2} , \| \cdot \|_M &{\upharpoonright}_{P \times \mathcal{P}(A)} \rangle \rightleftharpoons \\ 
 %       &\la B, {\sqsubset_N} {\upharpoonright}_{B^2} , \| \cdot \|_N {\upharpoonright}_{P \times \mathcal{P}(B)} \rangle.
  %  \end{align*}    
\end{definition}

    In the sequel we will omit the $P$ in the subscript and assume it to be the set of constants occurring in some `target' formula $\varphi$. 
As mentioned, bisimulations are useful because they preserve the truth of all $\mu$-calculus formulas, i,e.~if $\langle w,v\rangle\in\iota$ and $\varphi$ is any formula (with constants among $P$), then $w\in\|\varphi\|$ iff $v\in\|\varphi\|$.
As such, since every weakly transitive model is bisimilar to an irreflexive weakly transitive model, we will make the convention that every arbitrary model mentioned in this paper is irreflexive.

% \longversion{
% \begin{definition}[$\mu $-{\sf wK4}]
% The logic of $\mu $-{\sf wK4} is defined as the least set of formulas of $\mathcal{L}_{\mu}$ containing the following axioms and closed under the following rules (for all formulae $\varphi, \psi$ and $\theta(x)$ where $x$ occurs only positively in $\theta$):
% \begin{description}
%     \item All instances of axioms and rules of propositional logic closed under substitution for formulas in $\mathcal{L}_{\mu}$;
%     \item[Necessitation rule] $\dfrac{\varphi}{\Box \varphi}$;
%     \item[Distribution axiom(K)] $\Box(\varphi \to \psi) \to (\Box\varphi \to \Box \psi)$;
%     \item[Weak Transitivity] $\Diamond\Diamond \varphi \to \diamonddot \varphi$;
%     \item[Fixed point axiom] $\nu x .\theta(x) \to \theta(\nu x. \theta(x) )$;
%     \item[Induction rule] $\dfrac{\varphi \to \theta(\varphi)}{\varphi \to \nu x. \theta(x) }$
% \end{description}

% \end{definition}

% \begin{theorem}[Completeness, FMP \cite{BBFMu}]
%     $\mu$-{\sf wK4} is weakly complete for the class of weakly transitive frames. It also has the finite model property with respect to weakly transitive frames.
% \end{theorem}
% }

As a general rule, the $\mu$-calculus is more expressive than standard modal logic: for example, in a frame $(W,R)$, reachability via the transitive closure of $R$ is expressible in the $\mu$-calculus, but not in standard modal logic.
However, in the setting of transitive frames, reachability is already modally definable (since $R$ is its own transitive closure), which means that the familiar examples to show that the $\mu$-calculus is more powerful than modal logic do not apply.
Dawar and Otto~\cite{DO} exhibited an operator, since dubbed the {\em tangle,} which is $\mu$-calculus expressible but not modally expressible.
They showed the surprising result that every formula of the $\mu$-calculus can be expressed in terms of tangle.
In this paper, we will use a variant introduced by Gougeon \cite{Quent21}.
When working with multisets\footnote{By working with multisets, we can write $\vardiamond^\infty \{ \phi , \phi\} $ instead of $\vardiamond^\infty\{\phi , \phi \wedge \top\}$.}, if $x$ occurs $n$ times in $A$ then it occurs $\max\{0,n-1\} $ times in $A\setminus \{x\}$.

\begin{definition}
Given a finite multiset of formulae $\Gamma   \subseteq \mathcal{L}_{\mu}$, the tangle modality is defined as follows: 
\[ \vardiamond^{\infty}\Gamma = \nu x. \bigvee_{\varphi  \in \Gamma} \big( \diamonddot (\varphi  \wedge x) \land \bigwedge_{ \psi\in\Gamma\setminus\{\varphi\}} \Diamond (\psi \wedge x) \big), \]
where $x$ does not appear free in any $\varphi \in \Gamma$. 

We can then define the tangle logic $\mathcal{L}_{\vardiamond^\infty}$ whose language is defined by the syntax, where $\Gamma \subseteq_{fin} \mathcal{L}_{\vardiamond^\infty}$ is a multiset:
\[
    \varphi ::= \top \,  | \, p \, | \, \neg \varphi \, |\, \varphi \wedge \varphi \, |\,   \Diamond \varphi \,|\,   \vardiamond^\infty \Gamma .
    \]
\end{definition}

It can be checked that over transitive frames, $ \vardiamond^\infty \Gamma$ is equivalent to the `tangled derivative' $ \Diamond^\infty \Gamma$ \cite{GH17}, given by $\Diamond^{\infty}\Gamma := \nu x.   \bigwedge_{ \varphi \in \Gamma } \Diamond (\varphi \wedge x)  $.
The two   are also equivalent over familiar spaces such as the real line, but not over arbitrary topological spaces or weakly transitive frames, in which case $ \vardiamond^\infty  $ can define $\Diamond^\infty$ but not vice-versa~\cite{Quent21}.
In metric spaces such as the real line (and a wider class known as $T_D$ spaces), $\vardiamond^\infty \Gamma$ holds on $x$ if there is a perfect set $A$ (i.e., $A$ has no isolated points) containing $x$ such that for each $\varphi\in \Gamma$, $\val\varphi\cap A$ is dense in $A$.

\begin{example}
Consider a topological model based on the real line $\mathbb R$ with $\val r$ being the set of rational points and $\val i$ the set of irrational points.
Then, $\vardiamond^\infty\{r,i\}$ is valid on the real line, given that the sets of rational and irrational numbers are both dense.
In contrast, if we let $\val z$ be the set of integers, we readily obtain that $\vardiamond^\infty\{z,i\}$ evaluates to the empty set, given that the subspace of the integers consists of isolated points and hence we will not find any common perfect core between $\val z $ and $\val i$.
\end{example}

The tangle simplifies a bit when working over finite transitive frames.
In this case, this operator is best described in terms of clusters.
A cluster $C$ of a model $\mathcal{M}=\la M, \sqsubset, \| \cdot \|\ra$ is a subset of $M$ such that $\forall u,v\in C \, u \sqsubseteq v$.
Note that we don't define clusters to be maximal (with respect to set inclusion).
In contrast, {\em the} cluster of $w$ in $M$ is the set $C_w= \bigcup\{ C : C \text{ is a cluster of M and } w \in C \} $.

 It is well known that a transitive relation (and indeed even a weakly transitive relation) can be viewed as a partial order on its set of maximal clusters.
 To this end, define $w\prec v$ if $w\sqsubset v\not\sqsubset w$, and for $A,B \subseteq M$, we write:
\begin{itemize}
    \item $A \prec B$ iff $\forall v \in B\, \exists u \in A \, u \sqsubset v \not \sqsubset u $
    \item $A \preceq B$ iff $\forall v \in B\, \exists u \in A \, u \sqsubseteq v $.
\end{itemize}
Then, $\prec$ is a strict partial order on the maximal clusters of $M$.
In the sequel, $A, B$ will usually be nonempty clusters.
We also define e.g.~$w\prec A$ by identifying $w$ with $\{w\}$.

\begin{lemma}
Fix a multiset $\Gamma$ and a finite pointed model $(M,w)$, we have that $w\vDash_M \vardiamond^\infty \Gamma$ iff there is a cluster $C $ of $ M$ such that $w \preceq C$ and a map $f\colon C\to \Gamma$ such that $u\in\|f(u)\|$ for all $u \in C$, and whenever $\varphi\in\Gamma\setminus\{f(u)\}$, then there is $v\in C$ such that $u\sqsubset v \in C$ and $v \in \| \varphi \|$.
\end{lemma}

%\begin{proof}
 %   Immediate from the definition of the tangle.
%\end{proof}

% In practice what this means is that if there are $\varphi_i , \varphi_j \in \Gamma$ with $i \neq j$ and $\mu\text{-}{\sf wK4} \vdash \varphi_i \equiv \varphi_j $, then the cluster $C$ of the Lemma will be bisimilar to a cluster $C'$ such that $\exists v \in C'$ $v $ is reflexive and $ v \vDash_M \varphi_i$.\david{I don't understand this paragraph but would probably just delete.}

\begin{example}
Recall the model of Example~\ref{exmodel}, consisting of an irreflexive cluster $\{0,1\}$ with $\|e\| = \{0\}$, $\|o\| = \{1\}$, $\|p\| = \{1\}$, and $\|i\| = \{0,1\}$.
We then have that $\vardiamond^\infty \{e,o\} = \{0,1\}$, since each point is either even and has an accessible point that is odd, or vice-versa.
On the other hand, $\vardiamond^\infty \{o,p\} = \varnothing$, since we cannot assign any atom $a\in\{o,p\}$ to $1$ in such a way that $1$ satisfies $\diamonddot a\wedge\Diamond a' $, where $a' $ is the complementary atom to $a$. And if $0$ were to satisfy $ \diamonddot (a\land x) \wedge\Diamond (a'\land x)$, then $1$ would also have to satisfy $\diamonddot a\wedge\Diamond a' $, something we have already shown to be impossible.
Thus it is not enough for each element of $\Gamma$ to be satisfied in a cluster in order to make $\vardiamond^\infty \Gamma$ true: instead, each point $w$ must have an accessible world satisfying all but possibly one element $\varphi_w$ of $\Gamma$, in which case it must also satisfy $\varphi_w$.
\end{example}

\section{Final submodels}\label{sect: finality}

The technique of {\em final worlds} is a powerful tool in establishing the finite model property for many transitive modal logics~\cite{Fin74c}, and is also applicable to the $\mu$-calculus over weakly transitive frames~\cite{BBFMu}.
The idea here is that only a few worlds in a model contain `useful' information, and the rest can be deleted.
These `useful' worlds are those that are maximal (or {\em final}) with respect to $\sqsubseteq$, among those satisfying a given formula of $\Sigma$.

\begin{definition}[$\Sigma$-final]
Given a model $M$ and a set of formulas $\Sigma$, a world $w\in M$ is $\Sigma$-final if there is some formula $\varphi \in \Sigma$ such that $w \vDash_M \varphi$ and if $w\sqsubset u$ and $u \vDash_M \varphi$, then $u \sqsubset w$.\\
A set $A\subseteq M$ will be called $\Sigma$-final iff every $w\in A$ is $\Sigma$-final. The $\Sigma$-final part of $M$ is the largest $\Sigma$-final subset of $M$ and we denote it by $M^\Sigma$.
\end{definition}

Sometimes we need to `glue' a root cluster to a $\Sigma$-final model.
To this end, a rooted model $(M,w)$ will be called $\Sigma$-semifinal if $M\setminus C_w$ is $\Sigma$-final.

Baltag et al.~\cite{BBFMu} built on ideas of Fine~\cite{Fin74c} to show via final submodels that the topological $\mu$-calculus has the finite model property.
While final submodels are not necessarily finite (if $M$ is infinite), they do have finite {\em depth.}
  Given a model $M$, a set of formulas $\Sigma$ and $w
   \in M$, we define the {\em depth of $w$ in $M$,} denoted $dpt ^M (A) $, as the supremum of all $n$ such that $w=w_0
  \prec w_1\prec w_2\prec\ldots\prec w_n$ (recall that $\prec$ is the strict part of $\sqsubset$); note that this is finite on finite weakly transitive models but could be infinite on infinite ones.
For $A \subseteq M$ we define the depth of $A$ in $M$ to be $dpt  ^M(A) = \sup (0\cup \{    dpt  ^M (w):w\in A \})$.
  The $\Sigma$-depth of $w$ is defined analogously, except that here we only consider chains such that $w_1,\ldots,w_n\in M^\Sigma$ (note that $w$ itself need not be $\Sigma$-final).
  Then we define $dpt_\Sigma ^M(A)$ as before.
It is not hard to check that $dpt_\Sigma ^M(w)$ is bounded by $|\Sigma|$, and thus if $\Sigma$ is finite we can immediately control the depth of any $\Sigma$-final model. From a model of finite depth, it is easy to obtain a finite model.

In order to use this idea towards a proof of the finite model property (and also for our own results), one must carefully choose $\Sigma$ so that for any $\varphi\in \Sigma$ and $w\in M^\Sigma$, we have that $M^\Sigma ,w\equiv_\Sigma M,w$.
For example, $\Sigma$ should be closed under subformulas, but since we are in the $\mu$-calculus, we will have to find a way to treat the free variables that show up in said subformulas.
Because of this, we define a variant of the set of subformulas of a given formula where any free occurrence of a variable is labelled according to its binding formula, thus making sure that the same variable does not appear free with different meanings.
We also need to treat reflexive modalities as if they were primitive.

\begin{definition}
We define the modified subformula operator $sub^*:\mathcal{L}_{\mu}\to \mathcal{P}(\mathcal{L}_{\mu})$ recursively by
\begin{itemize}
    \item $sub^*(r) = \{r\}$ if $r = \top, p, x$;
    \item $sub^*(\neg p) = \{ \neg p , p\}$;
    \item $sub^*(\varphi \ocircle \psi) = \{ \varphi \ocircle \psi \} \cup sub^*(\varphi )\cup sub^* (\psi)$ where $ \ocircle = \wedge \text{ or } \vee$ and $\varphi \ocircle \psi \neq \diamonddot \sigma\text{  or } \boxdot \sigma$ for some $\sigma$;\footnote{Remember that $\diamonddot \sigma $ abbreviates $ \sigma \lor \Diamond \sigma$ and similarly  $\boxdot \sigma = \sigma \wedge \Box \sigma$. }
    \item $sub^*(\ocircle \psi) = \{ \psi \}\cup sub^*(\psi)$ where $\ocircle = \Diamond, \Box, \diamonddot $ or $\boxdot$;
    \item $sub^* (\nu x. \varphi) = \{ \varphi (x_{\nu x.\varphi})\}\cup sub^*(\varphi(x_{\nu x.\varphi})) $ where $x_{\nu x.\varphi}$ is a fresh propositional variable named after $\nu x. \varphi$;
    \item $sub^* (\mu x. \varphi) = \{ \varphi (x_{\mu x.\varphi})\}\cup sub^*(\varphi(x_{\mu x.\varphi})) $ where $x_{\mu x.\varphi}$ is a fresh propositional variable named after $\mu x. \varphi$.
\end{itemize}

Given a set of formulae $\Sigma$, we can define a partial order on $sub^*[\Sigma]$ by $\varphi <_{sub^*} \psi$ iff $\varphi \in sub^* (\psi)$ and $\varphi \neq \psi$.
\end{definition}
Observe that if $x_\psi$ is a free variable of $\varphi$, then $\varphi <_{sub^*} \psi $.
So we will work with these altered subformulas, but we also need to close $\Sigma$ under some further operations.
Given a set $\mathbb{X}$, some $Y \subseteq \mathbb{X}$ and a set $\mathcal{A}$ of mappings $a: \mathbb{X} \to\mathcal{P}(\mathbb{X}) $, we define the closure of $Y$ over $\mathbb{X}$ inductively as follows:
\begin{itemize}
    \item $Cl^0_\mathcal{A}(Y)=Y$;
    \item $Cl^{\alpha+1}_\mathcal{A}(Y)=Cl^\alpha_\mathcal{A}(Y)\cup \{a (x): a \in \mathcal{A} \ \& \, x \in Cl^\alpha_\mathcal{A}(Y)\}$;
    \item $Cl^\lambda_\mathcal{A}(Y)=\displaystyle\bigcup_{\alpha<\lambda} Cl^\alpha_\mathcal{A}(Y)$ for $\lambda \in Lim$.
\end{itemize}
$Cl_\mathcal{A}(Y) = Cl^\alpha_\mathcal{A}(Y)$ where $\alpha $ is any ordinal such that $Cl^\alpha_\mathcal{A}(Y)=Cl^{\alpha+1}_\mathcal{A}(Y)$.

For the remainder of the paper, unless stated otherwise, we will be working with a set of formulae $\Sigma$ such that $\Sigma = Cl_{\diamonddot, sub^*, \neg}(\Sigma)$. Observe that any finite set $\Sigma_0$ can be extended to a $\Sigma$ with this property that is finite up to modal equivalence of formulae since in $\sf S4$ there are only finitely many non equivalent modalities and $\diamonddot$ is an $\sf S4$ modality \cite{Chagrov1997ModalL}.

Since we have labelled our variables by their binding formula, we can substitute this formula back and obtain a `closed' version of this formula.
 
\begin{lemmaarep}
    Fix a finite set of formulas $\Sigma$ closed under $sub^*$ and some $\varphi \in \Sigma$, we let $\floor \varphi$ denote the closed form of $\varphi$; that is every instance of $x_\psi$ is substituted by $\psi$ recursively until there are no free variables left.\\
    It holds that $\floor \varphi \in \mathcal{L}_\mu$ for each $\varphi \in \Sigma$.
\end{lemmaarep}
\begin{proof}
    Inductively on the inverse of $<_{sub^*}$.\\
    Suppose $\varphi$ is such that there is no $\psi$ with $\varphi <_{sub^*} \psi $, then it has no free variables of the form $x_\psi$ and so $\lfloor \varphi \rfloor = \varphi$.\\
    Now given $\varphi$, assume by our induction hypothesis that $\lfloor \psi \rfloor \in \mathcal{L}_\mu$ for all $\varphi <_{sub^*} \psi$. Since for all $x_\psi$ in $\varphi$ it is the case that $\varphi <_{sub^*} \psi$, then $\lfloor \varphi \rfloor = \varphi [ x_\psi / \lfloor \psi \rfloor ]$ i.e. we substitute each $x_\psi$ showing up in $\varphi$ with $\floor \psi$.
\end{proof}

Observe that additionally $\neg \lfloor \varphi \rfloor$ is equivalent to $\lfloor \neg \varphi \rfloor$ for all $\varphi \in \Sigma$.
In the sequel, given a model $M$ and a set of formulae $\Sigma$ closed under $sub^*$, we will read $w \vDash_M \varphi$ to mean $w \vDash_M \floor \varphi$.
% \david{Is this a good idea?} 
In particular, this means that $w$ is final for $\varphi$ in $M$ iff it is final for $\floor \varphi$ in $M$.

\begin{definition}
Fix a finite rooted model $(M,w)$ and a set of formulas $\Sigma$, we will write 
\begin{align*}
    w \vDash_M \overline{\la n \ra} \varphi :\Leftrightarrow   \exists v \in M^{\Sigma} \, ( v \sqsupseteq w \wedge  dpt_{\Sigma}(v)=n  \wedge \,v\vDash_M \varphi ).
\end{align*}

\end{definition}
Since for a given cluster $C$ of $M$ and $u,v \in C, \, u\vDash_M \overline{\la n \ra}\varphi \Leftrightarrow v \vDash_M \overline{\la n \ra}\varphi$, we will occasionally make an abuse of notation and write $C \vDash_M \overline{\la n \ra}\varphi$ to mean $\exists u \in C \, u \vDash_M \overline{\la n \ra}\varphi$.
%\footnote{which is equivalent to $\exists u \in C \, u \vDash_M \overline{\la n \ra}\varphi$.}

The formulas $\overline{\la n \ra}\varphi$ provide all the information needed to evaluate truth on $C$:

\begin{theoremarep}\label{thm: depth with cluster equivalence}
Let $(M,w),(N,w)$ be finite rooted models with root clusters $C$ and $C'$ respectively. Assume that $dpt_{\Sigma}^M(w) = dpt_{\Sigma}^N(w)$ and $\forall \varphi \in \Sigma$ $w\vDash_M \overline{\la n \ra} \varphi \Leftrightarrow w \vDash_N \overline{\la n \ra} \varphi $ for all $n < dpt_{\Sigma}^M(w)$, and
\begin{itemize}
    \item if $C$ is $\Sigma$-final then $C'=C$
    \item if $C$ is not $\Sigma$-final then $C'\subseteq C$
\end{itemize}
then $\forall v \in C' \, \forall \varphi \in \Sigma \ v\vDash_M \varphi $ iff $v\vDash_N \varphi $.
\end{theoremarep}

% \longversion{
\begin{proof}
    Given a set of formulas $\Sigma \subseteq \mathcal{L}_\mu$, define
    \[
    \Sigma^\nu = \{ \sigma \in \Sigma : \sigma = \nu x. \varphi(x) \text{ or } \sigma = \mu x. \varphi(x) \text{ for some } \varphi \}.
    \]

Let $\vec X = (X_{\psi})_{\psi \in \Sigma^{\nu}_2}$ and $\vec Y = (Y_{\psi})_{\psi \in \Sigma^\nu}$ be tuples of sets such that $X_\psi \subseteq M$, $Y_\psi \subseteq N$, $X_\psi \cap (M \setminus  C') = \| \lfloor \psi \rfloor \|_M \setminus C'$, $Y_\psi \cap (N \setminus  C') = \| \lfloor \psi \rfloor \|_N \setminus C'$ and $X_\psi \cap C' = (Y_\psi \cup \| \psi(\vec X) \|_M) \cap C' $.\footnote{Here more formally, $\vec X$ only contains the $X_\sigma$ such that $x_\sigma$ occur as free in $\psi$.} Observe that the $X_\psi$ are well defined. We show this by induction on the inverse of $<_{sub^*}$ on the set $\Sigma ^\nu$. $\psi$ is such that there are no $\psi <_{sub^*} \sigma$, then $X_\psi = \| \psi \|_M \cup Y_\psi$ and it is well defined. Suppose that $\psi$ is such that $X_\sigma$ is well defined for all $\psi <_{sub^*} \sigma $, then since the only occurrences of free variables $x_\sigma$ in $\psi$ are for $\psi <_{sub^*} \sigma$, and so $X_\psi = \| \psi [ x_\sigma / X_\sigma] \| _M \cup Y_\sigma$ which is well defined.
\begin{description}
    \item[Claim:] For all $\psi \in \Sigma$ and $w \in C'$, if $w\vDash_N \psi(\vec Y)$ then $w\vDash_M \psi(\vec X)$.
    % \item[Claim 2] If instead $X_\psi = Y_\psi =\| \lfloor \psi \rfloor \|_M$, then for all $\psi \in \Sigma$ and $u \in C'$, $u\vDash_N \psi(\vec Y)$ iff $u \vDash_M \psi(\vec X)$.
\end{description}
This is proven by induction on the structure of the form of formulas $\psi$ that have no occurrences of the $\neg$ symbol outside of the atomic cases. The cases for the set variables, literals variables and logical connectives are immediate.\\
Case for $\varphi =\Diamond\sigma$:\\
If $w \vDash_N \varphi(\vec Y)$, then $\exists u \in N \ w \sqsubset_N u$ such that $u \vDash_N \sigma(\vec Y)$. If $u \in C'$, then by induction hypothesis $u \vDash_M \sigma(\vec X)$ and since $w \sqsubset_M u$, then $w \vDash_M \varphi(\vec X)$. If $u \not \in C'$, then $u \vDash_N \sigma(\vec Y) $ iff $v\vDash_N \lfloor \sigma \rfloor$, so there is $\Sigma$-final $u \sqsubset_N u'$ such that $u' \vDash_N \lfloor\sigma\rfloor$ and so $w \vDash_N \overline{\la n \ra} \lfloor \sigma\rfloor $ for some $n < dpt_{\Sigma}^N(w)$. Thus $w \vDash_M \overline{\la n \ra} \lfloor \sigma\rfloor \Rightarrow w \vDash_M \lfloor \varphi \rfloor \Rightarrow w \vDash_M \varphi(\vec X) $.\\
Case for $\varphi = \Box \sigma$:\\
If $w \vDash_N \varphi (\vec Y)$ then --as with the $\Diamond$ case-- by the induction hypothesis, for all $w\sqsubset_N u$ with $u \in C'$ as well as for all $w \prec_M u$ $u \vDash_M \sigma(\vec X)$. In particular, for all $n< dpt_{\Sigma}^M (w)$ $ w \not\vDash_M \overline{\la n \ra} \neg \lfloor \sigma \rfloor $. 
In the case where $C\setminus C' \neq \varnothing$, we have that $C$ is not $\Sigma$-final in $M$, then assume towards a contradiction that for some $u \in C\setminus C'$ $u \not\vDash_M  \sigma (\vec X)$. Then by monotonicity we get that $u \not \vDash_M \floor \sigma$ i.e. $u \vDash_M \neg \floor \sigma$, then since $u$ is not $\Sigma$-final, there is some $u \sqsubset_M u'$ such that $u' \vDash_M \neg \floor \sigma$, a contradiction since $dpt_{\Sigma} ^M (u') < dpt_{\Sigma} ^M (w)$. Thus in either case $u \vDash_M \sigma(\vec X)$ for all $w \sqsubset_M u$ and so $w \vDash_M \varphi(\vec X)$.
% for all $u \in C\setminus C' \, \exists u'$ $\Sigma$-final such that $ M,u \equiv_{\Sigma} M, u'$ and so $u' \vDash_M \sigma(\vec X) \Leftrightarrow u' \vDash_M \lfloor \sigma \rfloor \Leftrightarrow u \vDash_M \lfloor \sigma \rfloor \Rightarrow u \vDash_M \sigma (\vec X)$ by monotonicity. Thus $w \vDash_M \varphi (\vec X)$.
\\
The cases for the $\diamonddot$ and $\boxdot$ formulas are mutandis mutatis.\\ 
Case for $\varphi =\nu x. \sigma (x)$:\\
Let $Y_{\varphi } = \| \varphi (\vec Y) \|_N $, then $ \| \varphi (\vec Y) \|_N \cap C' = \| \sigma (\vec Y) \| _N \cap C' = Y_\varphi \cap C' \subseteq X_\varphi  $. We show that $X_\varphi \subseteq \sigma (\vec X) $. Let $u \in X_\varphi$
\begin{description}
    \item[Case 1:] $u \in Y_\varphi \Leftrightarrow u \in \| \sigma (\vec Y) \|_N \Rightarrow u \in \| \sigma (\vec X) \|_M $ by the induction hypothesis.
    \item[Case 2:] $u \in \| \varphi(\vec X) \|_M = \| \sigma (\varphi(\vec X) , \vec X) \|_M \subseteq \| \sigma (\vec X) \|_M $ by monotonicity.
\end{description}
Case for $\varphi = \mu x. \sigma (x)$:\\
Since $C'$ is finite
% \footnote{or at least bisimilar to a finite cluster.}, 
$\| \varphi(\vec Y)\|_N = \|  \sigma^{m}(x, \vec Y)\|_N$ for some $m$, where 
\begin{description}
    \item $\sigma^0(x,\vec Y) := \| \lfloor \varphi \rfloor \|_N \setminus C'$;
    \item $\sigma^{n+1}(x,\vec Y) := \sigma( \sigma^n(x,\vec Y), \vec Y)$.
\end{description}
We show inductively on $n$ that for all  $Y_\varphi = \|\sigma^n (x, \vec Y) \|_N $ that $C' \cap Y_\varphi \subseteq \| \varphi (\vec X) \|_M = X_\varphi$. This is trivially true for $Y_\varphi = \| \floor \varphi \| _N \setminus C'$. For the inductive step, suppose it holds for some $Y_\varphi = \|\sigma^{n}(x,\vec Y)\|_N $, then by the original induction hypothesis, $\| \sigma^{n+1}(x,\vec Y)\|_N \cap C' =\| \sigma (\vec Y) \| _N \cap C' \subseteq \| \sigma (\vec X ) \|_M $ and so for $Y_\varphi ' = \|\sigma^{n+1}(x,\vec Y)\|_N $, we get $X_\varphi ' = \| \varphi (\vec X)\|_M$.\\
This proves the Claim. Now, by letting $Y_\psi = \| \lfloor \psi \rfloor \| _N $ for all $\psi \in \Sigma ^{\nu}$, we get by a simple induction on the inverse of $sub^*$ that $X_\psi = \| \floor \psi \|_M$, and so combining this with the claim we get:
\[\forall\varphi \in \Sigma \ \| \floor \varphi \| _N \cap C' \subseteq \| \floor \varphi \|_M \cap C'.\]
Since $\Sigma$ is closed under negation, it is also the case that $\| \lfloor \neg \varphi \rfloor \|_N \cap C' \subseteq \| \lfloor \neg \varphi \rfloor \|_M \cap C' $, ie $C' \setminus \| \lfloor  \varphi \rfloor \|_N  \subseteq C' \setminus \| \lfloor  \varphi \rfloor \|_M $ and so  $ \| \lfloor \varphi \rfloor \|_M \cap C' \subseteq \| \lfloor  \varphi \rfloor \|_N \cap C'$, which proves the theorem.
\end{proof}
% }

As an immediate corollary, we get the following, where we write $M,u \equiv_{\Sigma}N,v$ to mean  $\forall \varphi \in \Sigma$ $u \vDash_M \varphi \Leftrightarrow v\vDash_N \varphi$. In case $M=N$, we may abbreviate this by $u\equiv_{\Sigma}v$.

\begin{theoremarep} \label{thm: in between model and its final part}
Given a finite model $M$, a model $N$ with $M \supseteq N \supseteq M^{\Sigma} $ and any $w \in N$, it holds that $M,w\equiv_{\Sigma} N,w$.
\end{theoremarep}

\begin{proof}
    By induction on the depth of the clusters of $N$. Clusters of depth $0$ are necessarily $\Sigma$-final and so the claim holds trivially true. Assume that the claim holds for all upwards closed submodels of $N$ of depth $m \leq n$, then by the assumption, a cluster of depth $m\leq n$ is then $\Sigma$-final in $M$ iff it is so in $N$, thus a cluster $C'$ of $N$ of depth $n+1$ satisfies the same $\overline{\la k \ra }\varphi$ formulae in $N$ as does its counterpart $C$ in $M$ for $k < dpt_\Sigma ^M (C) = dpt_\Sigma ^N (C') $. An application of Theorem \ref{thm: depth with cluster equivalence} concludes the proof.
\end{proof}

\section{Structural evaluation} \label{sect: structural evaluation}

The strategy we will follow to obtain an equivalence is to describe the parts of the world and the model that are relevant to Theorem \ref{thm: depth with cluster equivalence}. In particular we will define formulae in $\mathcal{L}_{\vardiamond^\infty}$ equivalent to the $\overline{\la n \ra}\varphi$ `formulae', as well as a formula which approximates the statement ``$w$ is $\Sigma$-final".

\longversion{
An alternative approach could have been to instead use only Theorem \ref{thm: in between model and its final part} and produce a formula that determines the $\Sigma$-final part of a given rooted model. While this approach would have also worked, it would have resulted in a super-exponential upper bound on the size of the induced formula (relative to the size of its equivalent $\mathcal{L}_\mu$ formula $\varphi$).
}

 Theorem \ref{thm: depth with cluster equivalence} tells us that we need very little information to evaluate truth of formulas on a given cluster, provided we have already evaluated them on clusters of lower depth.
 This information is recorded by (semi-)satisfaction pairs:

\begin{definition}
Given a model $M$ say that $\la C,\Theta \ra$ is a semi-satisfaction pair for $M$ if $\exists w \in M$ such that $C = C_w$ and $\Theta= \{ \overline{\la m \ra} \psi: w \vDash_M \overline{\la m \ra} \psi $ for $ \psi\in \Sigma \wedge m<dpt_{\Sigma}(w) \}$. A pair $\la C, \Theta\ra$ is called a semi-satisfaction pair if it is a semi-satisfaction pair for some finite pointed $ \Sigma$-semifinal model. A satisfaction pair for $M$ is a semi-satisfaction pair $\la C,\Theta\ra$ such that $C$ is $\Sigma$-final in $M$.

Given a semi-satisfaction pair $\la C,\Theta \ra$ for some model $M$, we define\footnote{Due to Theorem \ref{thm: depth with cluster equivalence}, $\Theta^C$ is uniquely determined irrespectively of the chosen model $M$ for which $\la C,\Theta \ra $ is a semi-satisfaction pair.}
\begin{equation*}
    \displaystyle\Theta^C := \{ \overline{\la m \ra} \psi: C \vDash_M \overline{\la m \ra} \psi \text{ for } \psi\in \Sigma \wedge m\leq dpt_{\Sigma}(C) \}.
\end{equation*}

We extend the definition of $dpt_{\Sigma} $ by saying $dpt_{\Sigma}(\Theta)= sup\{n: \overline{\la n \ra}\varphi \in \Theta$ for some $\varphi \in \Sigma\}$. Let $Sat_n$ be the set of satisfaction pairs $\la C , \Theta \ra$ such that $dpt_{\Sigma} (\Theta) = n$ and let $Sat^0_n, \, Sat^1_n$ be the first and second projections of $Sat_n$ respectively. Similarly, $Sat^*_n, Sat^{*0}_n, Sat^{*1}_n$ are the corresponding sets for semi-satisfaction pairs.
\end{definition}

We will need to compare clusters and semi-satisfaction pairs.
Roughly, $C\subsetplus C'$ indicates that $C$ is a smaller cluster than $C'$ (up to bisimulation), and $\la C , \Theta \ra \vartriangleleft \la C' , \Theta' \ra$ indicates that the two pairs vary only in their root cluster, where $C'$ is larger.

Let us make this precise.
Fix $P\subseteq \mathbb{P}$ and clusters $C $ and $ C'$ from models $\mathcal{M}= \la M, \sqsubset_\mathcal{M}, \| \cdot \|_\mathcal{M} \ra $ and $\mathcal{N}=\la N, \sqsubset_\mathcal{N}, \| \cdot \|_\mathcal{N} \ra$ respectively, we write $C \subsetpluseq_P C'$ to mean that there is some $C'' \subseteq  C''' $ such that $C' \rightleftharpoons_P C''$. Similarly $C\subsetplus C'$ is defined for when additionally $C \not \rightleftharpoons_P C'$. As with the bisimilarity notation, the $P$ subscript is omitted in the sequel.
Define $ \vartriangleleft_n \, \subseteq Sat_n \times Sat_n $ by $\la C' , \Theta' \ra \vartriangleleft_n \la C , \Theta \ra $ iff $C' \subsetplus C$ and $\Theta' = \Theta$. Let $\trianglelefteq_n$ be the reflexive closure of $\vartriangleleft_n$. We will write $\vartriangleleft$, $\trianglelefteq$ instead of $\vartriangleleft_n$, $\trianglelefteq_n$ when $n$ is clear.

Satisfaction pairs are sufficient to evaluate truth, but our definition of $\overline{\la n \ra}\varphi$ in tangle logic will be sensitive to depth (i.e., to $n$), and thus we need to control the $\Sigma$-depth of the model we are working in.
This is achieved by considering chains of satisfaction pairs: if a chain of length $n$ lies above a given world, that means that the depth of that world is at least $n$.
Since the property `there is a chain of length $n$' will be expressible in $\mathcal L_{\vardiamond^\infty}$, this will allow us to have the desired control over depth.

To formally define chains, we need to consider root clusters glued to a model.
Fix a finite model $M$ and a cluster $C$ with $M\cap C = \varnothing$, we denote by $\left[ \ontop{M}{C} \right]$ the model $N$ with domain $M\cup C$, accessibility relation $\sqsubset_N := {\sqsubset_M}\cup {\sqsubset_C} \cup {(C\times M)}$ and $\|\cdot\|_N := \|\cdot\|_M\cup\|\cdot\|_C$. 

\begin{lemma} \label{lem: witnessing chains}
     For every $\Sigma$-final model $M$ of depth $n$ with a root cluster $C$, there is some chain $\mathcal{C}=\{ \la C_i , \Theta_i \ra \}_{i\leq n}$ such that 
\begin{enumerate}
    \item $C_n = C$
    \item $\la C_i, \Theta_i \ra$ is a satisfaction pair for $M$ for each $i\leq n$
    \item $C_{i+1} \prec C_i$ for each $i <n$
    \item \label{item: describable chain condition} For all $i<n$, if $\left[ \ontop{C_i}{C_{i+1}} \right] \rightleftharpoons C_i$  then $\Theta_{i+1} \neq \Theta_i ^{C_i}$.
\end{enumerate}
\end{lemma}

\longversion{
\begin{proof}
    We need to show that we can always choose a chain as above such that the clause \ref{item: describable chain condition} above holds. Towards a contradiction, let $M$ be a $\Sigma$-final model with a root cluster $C$ of depth $n+1$ that belongs to a satisfaction pair $\la C,\Theta\ra$ for $M$ and such that for every cluster satisfaction pair $\la C' , \Theta' \ra $ of depth $n$, then $ \left[ \ontop{C'}{C} \right] \rightleftharpoons C'$ and $\Theta=\Theta'^{C'}$.
    Then by Theorem \ref{thm: depth with cluster equivalence} $ \left[ \ontop{C'{\uparrow^*_M}}{C} \right]$ and $M$ satisfy the same $\Sigma$ formulae on $C$. However $\left[ \ontop{C'{\uparrow^*_M}}{C} \right] \rightleftharpoons C'{\uparrow^*_M}$ which implies that $C$ is not $\Sigma$-final in $M$, a contradiction.
\end{proof}
}

    We will call a chain as in Lemma \ref{lem: witnessing chains} a {\em witnessing chain of depth $n$}; witnessing chains will be denoted as $\mathcal{C}$, $\mathcal{C'}$ or $\mathcal{C}_i$.
    Let $Chain_n$ be the set of witnessing chains of depth $n$. We extend $\vartriangleleft_n$ to $Chain_n \times Chain_n$ by setting $\mathcal{C}\vartriangleleft_n \mathcal{C}'$ iff the following hold:
    \begin{itemize}
        \item $\la C_i, \Theta_i\ra = \la C_i',\Theta_i'\ra$ for $i<n$
        \item $C_n  \subsetplus C_n'$
        \item $\Theta_n = \Theta_n'$
    \end{itemize}
    and let $\trianglelefteq_n$ be its reflexive closure. We will identify $\vartriangleleft$ and $\trianglelefteq$ to be the appropriate $\vartriangleleft_n$ and $\trianglelefteq_n$ respectively. Finally, given $n$ and a formula $\varphi \in \Sigma$, we write
\begin{align*}
    supp(\overline{\la n \ra } \varphi)= 
    \big\{ \mathcal{C} \in Chain_n : \exists (M,w) \text{ finite pointed }
    \Sigma\text{-final model where } \\
    w\vDash_M \varphi \wedge \, C_n=C_w \, \wedge 
    \mathcal{C} \text{ is a witnessing chain of depth } n \text{ for } M \big\}.
\end{align*}

The definition of witnessing chains can be further expanded to semifinal models, however the analogue of Lemma \ref{lem: witnessing chains} for semi-witnessing chains will not necessarily hold for any $\Sigma$-semifinal model as we cannot guarantee that we can always find a chain in that case for which condition \ref{item: describable chain condition} will hold for the root cluster.
In this setting, we instead use a weaker notion.

\begin{definition}
Given a $\Sigma$-semifinal model $M$ of depth $n$ with root cluster $C$, a semi-witnessing chain for $M$ of depth $n$ (if it exists) is some chain $\mathcal{C}=\{ \la C_i , \Theta_i \ra \}_{i\leq n}$ such that 
\begin{enumerate}
    \item $C_{n} = C$
    \item $\la C_i, \Theta_i \ra$ is a semi-satisfaction pair for $M$ for each $i\leq n$
    \item $C_{i+1} \prec C_i$ for each $i <n$
    \item For all $i<n$, if $\left[ \ontop{C_i}{C_{i+1}} \right] \rightleftharpoons C_i$  then $\Theta_{i+1} \neq \Theta_i ^{C_i}$.
\end{enumerate}

We will denote by $Chain_n^*$ the set of all semi-witnessing chains of depth $n$.
For $M$ an arbitrary finite model, a (semi-)witnessing chain on $M$ of depth $n$ will be a (semi-)witnessing chain on the $\Sigma$-(semi)final part of $w\uparrow^*_M$ for some $w\in M$. Finally for $\mathcal{C} \in Chain^*_n$, let $dpt(\mathcal{C}):=n$ denote its depth.
\end{definition}

\begin{figure}[ht]
    \centering
    \begin{tikzpicture}[every node/.style={font=\footnotesize}]
    \node[state]    (c1)   at (0,0)  {$C_n$};
    \node[state]    (c2)   at (-0.3,1.2)  {$C_{n-1}$};
    \node[state]    (c3)   at (0.7,2.4) {$C_{n-2}$};
  \path[draw] (-2.2,4) -- (-0.37,-0.17);
    \draw (2.2,4) -- (0.36,-0.17);
    \path[->]       (c1) edge node {} (c2)
                    (c2) edge node {} (c3);
    \path[-] (c3) edge [snake it] node {} (0,4);

    \node (w1) at (4.7,0) {w};
    \node[state]    (a3)   at (4.9,1.7) {$\la n \ra \varphi$};
    \node (text) at (4.7,3.2) {$\mathcal{C}$};
    \path[draw] (2.6,4) -- (w1);
    \draw (6.8,4) -- (w1);
    \path[->]       (w1) edge  node {} (a3);
    \path[-] (a3) edge [snake it,] node {} (4.1,4);
\end{tikzpicture}
    
    \caption{On the left, a witnessing chain. On the right, a witnessing chain ensures that the $\Sigma$-depth of a point where $\la n \ra \varphi$ holds is at least $n$.}
    \label{fig:chains}
\end{figure}
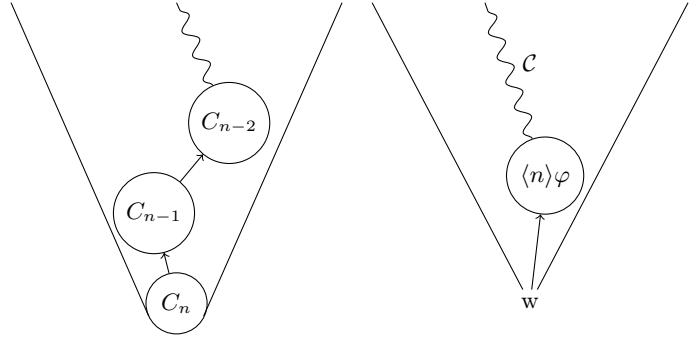

We can now define formulas equivalent to the $``\overline{\la n \ra} \varphi "$ in the language of $\mathcal{L}_{\vardiamond^\infty}$. This is done inductively by having the formula $\alpha$ express the existence of a witnessing chain $\mathcal{C}$ with a satisfaction pair $\la C,\Theta\ra$ underneath it. Then the formulae $\beta$ and $\gamma$ ensure that the extension $\mathcal{C}^\frown \la C,\Theta\ra$ is also a witnessing chain (i.e. the pair $\la C,\Theta \ra$ is as high as it can possibly be while remaining below $\mathcal{C}$). At this point it is important to note that if we were to simply use satisfaction pairs, we would run the risk of having the $\Sigma$-depth of worlds satisfying $\la n \ra \varphi$ being smaller than $n$; with witnessing chains, we ensure that the depth does not collapse.

\begin{definition}
    Fix $w\in M$ and a set of formulae $\Sigma$ let $\tau_w := \bigwedge_{p \in P(w)} p \wedge \bigwedge_{p\not \in P(w)}\neg p$, where $p\in \Sigma$. We will, as a convention, not include the model $M$ and the set $\Sigma$ in the notation.
    Below we define the formulas $\la n \ra \varphi\in\mathcal L_{\vardiamond^\infty}$, along with some auxiliary formulas and notation.
\begin{itemize}
    \item $Ir(\mathcal{C}):= \la C_{n},\Theta_{n}\ra  \trianglelefteq \la C_{n-1},\Theta _{n-1}^{C_{n-1}}\ra \wedge \exists w \in C_{n} \forall u \in C_{n} \cap w \uparrow \ P(w) \neq P(u)$ where $n=dpt(\mathcal{C})$
    \item $\displaystyle A(\Theta) := \bigwedge_{\overline{\la m \ra} \psi \in \Theta} \la m \ra\psi \wedge \bigwedge_{\overline{\la m \ra} \psi \not\in \Theta} \neg \la m \ra\psi$
    \item $
    \displaystyle\tau_w ^\mathcal{C} := 
    \begin{cases}
        \tau_w \wedge A(\Theta_{dpt(\mathcal{C})}) \wedge \Diamond \big( \tau_w \wedge \delta(\mathcal{C}{\upharpoonright} dpt(\mathcal{C})) \big) & \text{ if }  Ir(\mathcal{C})\\
        \tau_w \wedge A(\Theta_{dpt(\mathcal{C})}) \wedge \Diamond\delta(\mathcal{C}{\upharpoonright} dpt(\mathcal{C})) & \text{ otherwise }
    \end{cases} $
    \item $\alpha(\mathcal{C}):= \vardiamond^\infty \{ \tau_w ^\mathcal{C}: w \in C_{dpt(\mathcal{C})}\}$
    \item $\beta(\mathcal{C}) :=  \Box \big( \displaystyle \bigvee_{\mathcal{C}' \vartriangleleft \mathcal{C}} \alpha(\mathcal{C}') \to \alpha(\mathcal{C}) \big) $
    \item $\gamma(\mathcal{C}) := \displaystyle \neg \bigvee _{\mathcal{C}' \not \trianglelefteq \mathcal{C}} \alpha(\mathcal{C}')$
    \item $\delta(\mathcal{C}) := \alpha(\mathcal{C}) \wedge \beta (\mathcal{C}) \wedge \gamma (\mathcal{C})$
    \item $\la n \ra \varphi := \displaystyle \bigvee_{\mathcal{C} \in supp(\overline{\la n \ra } \varphi)} \diamonddot \delta(\mathcal{C})$
\end{itemize}
\end{definition}

Here, $A$ describes the $\overline{\la m\ra}$-formulas in a given $\Theta$, $Ir$ tells us when a bottom-most cluster in a chain has an `irreflexive point'\footnote{Whilst by our convention every world $w$ in $M$ is irreflexive, in this context we mean that $C_n,w \not\rightleftharpoons C',w'$ with $w'$ being reflexive.} which we can use to be able to jump to cluster in the chain above it, $\tau^\mathcal C_w$ describes the `local state' at $w$, $\alpha$ ensures that the desired chain is present, and $\beta$ and $\gamma$ rule out any unwanted chains.
By following step by step the definitions above, we can prove the following lemma:
\begin{lemmaarep} \label{lem: equivalence of the <n>phi definitions}
Fix a finite model $M$ a set of formulas $\Sigma$ and $w \in M$, it holds that $w\vDash_M \la n \ra \varphi \Leftrightarrow w \vDash_M \overline{\la n \ra} \varphi$ for all $\varphi \in \Sigma$.
\end{lemmaarep}

% \longversion{
\begin{proof}
By induction on $n$.\\
Assume that the statement holds true for all $m<n$, and let $\mathcal{C} \in supp(\overline{\la n \ra }\varphi)$. We claim the following:
\begin{enumerate}
    \item \label{item: alpha claim} $w\vDash_M \alpha(\mathcal{C})$ iff $ \mathcal{C}{\upharpoonright} n$ is a witnessing chain of depth $n$ for $M$ strictly above $w$ (i.e. $w\prec C_{n-1}$) and there is some cluster $C = C_u$ for some $u\in M$ such that
    \begin{enumerate}
        \item $w \preceq C \prec C_{n-1}$;
        \item $C_n \subsetpluseq C$;
        \item $C \vDash_M A(\Theta_n)$.
    \end{enumerate}
    \item \label{item: delta claim} $w \vDash_M \diamonddot\delta(\mathcal{C}) $ iff $\mathcal{C}$ is a witnessing chain of depth $n$ for $M$ above $w$ (i.e. $w \preceq C_n$).
\end{enumerate}
We prove the claims by the same induction on $n$. From the definition of $\alpha(\mathcal{C})$ and by the I.H., it should be clear that $C_n \subsetpluseq C$ and $C \vDash_M A(\Theta_n)$    . By Claim \ref{item: delta claim}, $ \mathcal{C}{\upharpoonright} n$ is a witnessing chain of depth $n$ for $M$ above the cluster $C$ of Claim \ref{item: alpha claim} (i.e. $C\preceq C_{n-1}$). We show that it is strictly above $C$.

\begin{itemize}
    \item If $C\not\subsetpluseq C_{n-1}$ then by Claim \ref{item: delta claim} $C \prec C_{n-1}$;
    \item If $ C \subsetpluseq C_{n-1}$ and  $C_{n-1}\not \vDash_M A(\Theta_n) $, then $C \prec C_{n-1}$ as $C\vDash_M A(\Theta_n)$;
    \item If $ C \subsetpluseq C_{n-1}$ and $C_{n-1} \vDash_M A(\Theta_n)$, then let $N=\left[ \ontop{C_{n-1}{\uparrow^*_M}}{C_{n}} \right]$. Then since $C_{n-1}{\uparrow^*_M}$ is a model admitting the witnessing chain $\mathcal{C}{\upharpoonright} n$, then so is $N$ for $\mathcal{C}$. Then it must be the case that $Ir(\mathcal{C})$ as otherwise $N \rightleftharpoons C_{n-1}{\uparrow^*_M} $ which would contradict the $\Sigma$-finality of the cluster $C_n$ in $N$. Then any $v$ witnessing $Ir(\mathcal{C})$ would also be irreflexive in $C$ in $M$ and so $C \prec C_{n-1}$ from the definition of $\tau^\mathcal{C}_w$.
\end{itemize}
This proves Claim \ref{item: alpha claim}. Now observe that in addition to that, if the cluster $C$ of Claim \ref{item: alpha claim} has $dpt_{\Sigma}(C)=n$, then it is a $\Sigma$-final cluster. If not, then by Theorem \ref{thm: in between model and its final part} the model $N= \big(M \setminus (C\setminus C_n) \big) $ has the cluster $C_n$ not be $\Sigma$-final which in turn contradicts Theorem \ref{thm: depth with cluster equivalence}.

Now assume that the representative cluster $C$ for $C_n$ is the topmost cluster in $M$ in which $\alpha(\mathcal{C})$ holds. To prove Claim \ref{item: delta claim}, we first show that assuming $\diamonddot \delta (\mathcal{C})$, there is no $\Sigma$-final $C_v$ with $C_n \prec C_v \prec C_{n-1}$ and $dpt_{\Sigma}(C_v) = n$. Let $\mathcal{C}' = \mathcal{C}{\upharpoonright}n \cup \la n, \la C_v,\Theta \ra \ra$ for $\Theta$ such that $\mathcal{C}'$ is a witnessing chain of depth $n$ for $M$.

\begin{itemize}
    \item Assume $\mathcal{C}' \vartriangleleft \mathcal{C}$, i.e. $C_v \subsetplus C_n$ and $\Theta=\Theta_n$. By the clause of  $\beta(\mathcal{C})$, the clause of $\alpha(\mathcal{C})$ must hold above $C_v$. But this contradicts the choice of $C$.
    \item If $\mathcal{C}'\not \trianglelefteq \mathcal{C}$, then we contradict the clause of $\gamma (\mathcal{C})$.
\end{itemize}
For the reverse direction, let $\mathcal{C}$ a witnessing chain of depth $n$ for $M$ above $w$. We will show that for given $u \in C_n\, u\vDash_M \delta(\mathcal{C})$. It is clear that $\alpha(\mathcal{C})$ holds and additionally since $dpt_{\Sigma}(C_n)=n$ and $C_n$ is $\Sigma$-final, there are no other eligible witnessing chains $\mathcal{C}'$ with $\mathcal{C}' \vartriangleleft \mathcal{C}$ or $\mathcal{C}' \trianglelefteq\mathcal{C}$ above $u$ in $M$ and so $\beta(\mathcal{C})$ and $\gamma(\mathcal{C})$ hold as well.

Finally $w \vDash_M \overline{\la n \ra}\varphi $ iff there is some witnessing chain $\mathcal{C} \in supp(\overline{\la n \ra }\varphi) $ above $w$ iff $w \vDash_M \la n \ra \varphi$.
\end{proof}
% }

\begin{corollary}
Fix a finite model $M$, some $w \in M$ and $C\in Chain_n^* \setminus Chain_n$, then $w\vDash_M \alpha(\mathcal{C})$ iff $ \mathcal{C}{\upharpoonright} n$ is a witnessing chain of depth $n$ for $M$ strictly above $w$ (i.e. $w\prec C_{n-1}$) and there is some cluster $C = C_u$ for some $u\in M$ such that
    \begin{enumerate}[label=(\alph*)]
        \item $w \preceq C \prec C_{n-1}$
        \item $C_n \subsetpluseq C$
        \item $C \vDash_M A(\Theta_n)$
    \end{enumerate}
\end{corollary}

\longversion{
\begin{proof}
The proof is as in the lemma above.
\end{proof}
}

The formulas $ \la n \ra \varphi$ thus defined are the central ingredient in proving our main result.
The translation $\chi(\varphi)$ of $\varphi$ itself into $\mathcal L_{\vardiamond^ \infty}$ requires a case distinction according to whether we are evaluating on a final world or not.
Since a completely accurate definition of finality is impossible to obtain, even in $\mathcal{L}_\mu$, we will instead approximate one with the following.
The formula $split(n) $ roughly states that there are two incomparable final worlds of depth $n$ above $w$, or there is a semi-witnessing chain of depth higher than $n$ above $w$; in either case, $w$ itself cannot be a final world of depth $n$.

\begin{definition}
We define formulas
\begin{align*}
    split(n) :=
    \displaystyle\bigvee \{    \diamonddot  \delta(\mathcal{C}) \wedge \diamonddot \delta(\mathcal{C}'): \mathcal{C}, \mathcal{C}' \in Chain_n \text{ with } \la C_n, \Theta_n\ra \neq \la C_n',\Theta_n'\ra \}
     \\
    \vee \bigvee \{\alpha(\mathcal{C}_0): \mathcal{C}_0\in Chain_{n+1}^* \setminus Chain_{n+1} \}.
\end{align*}
\end{definition}

Now, suppose we have access to the valuation at $w$, a chain $\mathcal C$ witnessing that $w$ is $\Sigma$-final of depth $n$ (with $\mathcal C=\varnothing$ if $w$ is not $\Sigma$-final), as well as the set $\Theta$ of formulas $\la m\ra\varphi$ with $m<n:= dpt_\Sigma(w)$ which are true on $w$.
For such a tuple $(w,\mathcal C,\Theta,n)$, we define a formula $\chi_0(w,\mathcal{C},\Theta,n )$ stating the above-mentioned properties, depending on whether $split(n)$ holds on $w$:
\begin{align*}   
    \displaystyle\chi_0(w,\mathcal{C},\Theta,n ):= 
    \begin{cases}
        \la n \ra \top \wedge \neg \la n+1 \ra \top \wedge\\
        \neg split(n) \wedge \tau_w \wedge \diamonddot\delta(\mathcal{C}) & \text{if } \mathcal{C} \neq \varnothing \\
        \la n \ra \top \wedge \neg \la n+1 \ra \top \wedge \\
        split(n) \wedge \tau_w \wedge A(\Theta) & \text{if } \mathcal{C} = \varnothing 
    \end{cases}.
\end{align*}

We are almost ready to define $\chi(w)$.
To do so, we first define $eval(\varphi,n)$ to be the set of all triples $\la w, \mathcal{C}, \Theta\ra $ for which there exists a rooted $\Sigma$-semifinal model $(M,w)$ such that
\begin{enumerate}
    \item $w\in M$
    \item $w\vDash_M \varphi$
    \item $\Theta = \{ \overline{\la m \ra}\psi : w\vDash_M \overline{\la m \ra }\psi \text{ for } \psi \in \Sigma \wedge m < dpt_{\Sigma}(w)\} $
    \item If $w \not \in M^{\Sigma}$ then $dpt_{\Sigma} (\Theta)=n $ and $\mathcal{C} = \varnothing$
    \item If $w \in M^{\Sigma}$ then $dpt_{\Sigma} (\Theta)=n-1 $ and $\mathcal{C}$ is a witnessing chain for $M$ of depth $n$ with $\la C_w,\Theta\ra = \la C_n,\Theta_n\ra$.
\end{enumerate}

And let $eval(\varphi) := \displaystyle \bigcup_n eval(\varphi,n)$.
Since $w$ satisfies $\varphi$ if and only if we can find $\mathcal C $ and $\Theta$ such that $\la w, \mathcal{C}, \Theta\ra \in eval(\varphi)$, we may define the characteristic formula $\chi(\varphi)$ of $\varphi$ by
\begin{align*}
    \chi(\varphi) := \bigvee_{\la w,\mathcal{C},\Theta\ra \in eval(\varphi)} \chi_0 \big(w,\mathcal{C},\Theta, dpt_{\Sigma}(\Theta)\big).
\end{align*}

\begin{toappendix}
\begin{lemma}\label{lem: finality detector}
Given a finite rooted model $(M,w)$ such that $w\vDash_M \la n \ra \top \wedge \neg \la n+1 \ra \top$ the following hold:
\begin{enumerate}
    \item \label{item: is not final} If $w\vDash_M split(n)$ then $w$ is not $\Sigma$-final in $M$.
    \item \label{item: is almost final} If $w\not\vDash_M split(n)$ then for every cluster $C\in M^{\Sigma}$ of depth $n$ and for $N =  w\cup C{\uparrow^*_M}$, the following hold:
    \begin{itemize}
        \item  $\forall \varphi \in \Sigma \ w \vDash_M \varphi \Leftrightarrow w\vDash_N \varphi$
        \item $N \rightleftharpoons C{\uparrow^*_M}$.
    \end{itemize}
\end{enumerate}
\end{lemma}

% \longversion{
\begin{proof}
Claim \ref{item: is not final} is derived from the claims proved in Lemma \ref{lem: equivalence of the <n>phi definitions} as $split(n)$ holds if one of the following holds:
\begin{enumerate}[label=(\alph*)]
    \item There are two satisfaction pairs in $M$: $\la C,\Theta \ra \neq \la C', \Theta' \ra \in Sat_n$ with $C,C'$ being above $w$ and hence $dpt_{\Sigma}(w)>n$.
    \item There is a semi-witnessing chain $\mathcal{C'}$ of $M$ of depth $n+1$ above $w$ and hence $dpt_{\Sigma}(w) \geq n+1$.
\end{enumerate}
For Claim \ref{item: is almost final} assume that $w$ is not final. Since $w\not \vDash_M split(n)$, none of the above two cases hold, therefore
\begin{itemize}
    \item Every satisfaction pair of depth $n$ in $M$ is the same.
    \item For every semi-satisfaction pair $\la C,\Theta \ra$ of $M$ above $w$ with $dpt_{\Sigma}(C)=n+1$ and every satisfaction pair $\la C',\Theta'\ra$ of $M$ with $dpt_{\Sigma}(C')=n$ and $C' \succ C$, it holds that $\left[ \ontop{C'}{C}  \right] \rightleftharpoons C' $ and $\Theta = \Theta'^{C'}$.
\end{itemize}

\end{proof}
% }
\end{toappendix}

\begin{theoremarep}\label{thm: characteristic formula}
Given a formula $\varphi$ and a finite rooted model $(M,w)$, we have that $w \vDash_M \varphi \Leftrightarrow w \vDash_M \chi(\varphi)$.
\end{theoremarep}
\begin{proof}
Let $n = dpt_{\Sigma}^M(w)$ and $\Theta = \{ \overline{\la m \ra}\psi : w\vDash_M \overline{\la m \ra }\psi \text{ for } \psi \in \Sigma \wedge m < dpt_{\Sigma}(w)\}$. Assume first that $w \vDash_M \varphi$ and consider the following cases.
\begin{enumerate}
    \item Suppose that $w \in M^{\Sigma}$, then $w \vDash_M \neg split(n)$  and $w\ \vDash_M \la n \ra \top \wedge \neg \la n+1 \ra \top$. Since $w \vDash_M \varphi$ then by the definition of $eval$ and the fact that $w \in M^{\Sigma}$, we have $\la w, \mathcal{C} , \Theta\ra \in eval (\varphi,n)$ for some witnessing chain for $M$ rooted at $w$ and so $ w\vDash_M \chi_0 (w, \mathcal{C},\Theta, n)$.
    \item Assume now that $w \not \in M^{\Sigma}$, then $w\vDash_M \la n-1 \ra \top \wedge \neg\la n \ra \top$ and
    \begin{enumerate}
        \item If $w\vDash_M split(n-1)$ then $(M,w)$ is a witness of $\la w , \varnothing,\Theta\ra \in eval(\varphi,n)$ and so $w \vDash_M \chi_0 (w, \varnothing,\Theta, n-1)$.
        \item If $w \not \vDash_M split(n-1)$, let $C$ a $\Sigma$-final cluster with $dpt_{\Sigma}(C)=n-1$. By Lemma \ref{lem: finality detector}, $w$ satisfies the same $\Sigma$-formulae as some $u\in C$ and since $u\vDash_M \tau_w \wedge \diamonddot \delta(\mathcal{C})$ for some $\mathcal{C}$, then so does $w$.
    \end{enumerate}
\end{enumerate}
Assume now that $w\vDash_M \chi (\varphi)$ and consider the same cases as before:
\begin{enumerate}
    \item If $w \in M^{\Sigma}$, then $w\vDash_M \chi_0 (w, \mathcal{C},\Theta, n) $ for some witnessing chain $C$ and so $C_w = C_n$ and by Theorem \ref{thm: depth with cluster equivalence} $w\vDash_M \varphi$.
    \item If $w\not \in M^{\Sigma}$, then $w\vDash_M \chi_0 (w ,\mathcal{C},\Theta, n-1) $ and
    \begin{enumerate}
        \item If $w\vDash_M split(n-1)$ then since $\la w , \varnothing,\Theta\ra \in eval(\varphi,n)$, by Theorem \ref{thm: depth with cluster equivalence} $w\vDash_M \varphi$.
        \item If $w \not \vDash_M split(n-1)$, then since $w\vDash \diamonddot \delta (\mathcal{C})$, there is a $\Sigma$-final cluster $C$ above $w$ that is the root of the witnessing chain $\mathcal{C}$ in $M$. By Lemma \ref{lem: finality detector} $w$ satisfies the same $\Sigma$-formulae as some $u\in C_{n-1}$, and since $u\vDash_M \diamonddot \delta (\mathcal{C})$, then $w\vDash_M \varphi$.
    \end{enumerate}
\end{enumerate}
\end{proof}

In view of~\cite{BBFMu}, this also applies to the class of topological spaces.
Moreover, $\vardiamond^\infty \Gamma$ can be expressed by a first order formula in all path-finite weakly transitive frames, where path-finite means that the ordering $\prec$ and its inverse $\prec ^{-1}$ are well-founded. So we get a first order expressibility of $\mathcal{L}_\mu$ in frames analogous to the ones in \cite{DO}.
Thus we obtain the following.

\begin{theoremarep}
$\mathcal{L}_{\mu} \equiv \mathcal{L}_{\vardiamond^\infty}$ over the class of topological spaces and the class of weakly transitive frames, and so $\mathcal{L}_\mu \subset {\sf FOL}{/}{\rightleftharpoons}$ over finite and path-finite weakly transitive frames.
\end{theoremarep}

\begin{proof}
    Immediate from the above remark, the finite model property of $\mathcal{L}_\mu$ over {\sf wK4} frames and from Theorem \ref{thm: characteristic formula}.
\end{proof}

In-fact, we fail to get a characterization theorem for the $\mu$ calculus over finite and path-finite weakly transitive frames. We show this via a bisimulation invariant formula of {\sf FOL} whose modal class is not definable via a $\mathcal{L}_\mu$ formula.

\begin{theoremarep}
    $\mathcal{L}_\mu \subsetneq {\sf FOL}{/}{\rightleftharpoons}$ over finite and path-finite weakly transitive frames.
\end{theoremarep}
\begin{proof}
    Consider the following formulae: 
    \begin{itemize}
        \item $\psi (x) := P(x) \to \exists y \sqsupset x \, \exists z \sqsupset x \, P(y) \wedge \neg P(z)$
        \item $\sigma_0 (x) := \neg P(x) \to \forall y \sqsupset x  \,\big(\neg P(y) \to \neg y \sqsubset x\big) $
        \item $\sigma_1 (x) := \neg P(x) \to  \exists y \, \exists z \, \big(x \sqsubset y \wedge y \sqsubset z \wedge z \sqsubset x \wedge P(y) \wedge P(z)\big)$
    \end{itemize}
    and let $\varphi(w) := P(w) \wedge \forall x \sqsupseteq w \big(\psi(x) \wedge \sigma_0 (x) \wedge \sigma_1 (x)\big)$.\footnote{Though we use the equality symbol in the definition of $\varphi$, it can be easily omitted.} We will first show that $\varphi(w)$ is invariant under bisimulations over finite frames. So let $M,w \rightleftharpoons N,w'$ be finite frames such that $M \vDash \varphi(w)$. We will show that so does $N$.
    
    Assume that $N \not \vDash \varphi(w)$, thus $\exists x' \sqsupseteq w$ such that $N \not \vDash \psi(x') \wedge \sigma_0 (x') \wedge \sigma_1 (x')$ and let $x \in M$ such that $M,x \rightleftharpoons N,x'$.
    \begin{description}
        \item[Case 1:] $N \vDash P(x')$ and so $N\vDash \neg \psi (x')$ thus $N\vDash \forall y \sqsupset x' \, \forall z \sqsupset x' \big(P(y) \leftrightarrow P(z)\big)$ however since $M \vDash \psi(x)$ there are $y,z \sqsupset x$ where $P(y) \wedge P(z)$ leading to a contradiction since $M,N$ are bisimilar.
        \item[Case 2:] $N\vDash \neg P(x')$ and $N \vDash \neg \sigma_0 (x')$. Thus $N \vDash \exists y' \sqsupset x' \, \big(x' \sqsupset y' \wedge \neg P(y)\big) $, however since $M \vDash \sigma_0 (v)$ for all $v \in M$, by bisimilarity we can then find an infinite path in $M$ of points $z$ such that $\neg P(z)$, contradicting our choice of models.
        \item[Case 3:] $N\vDash \neg P(x')$ and $N \vDash \neg \sigma_1 (x')$. Assuming that $N \vDash \sigma_0(x')$, we can conclude that there is at most one $y \in C_{x'}$ such that $N \vDash P(y)$. Then following the path $x \to y \to z$ in $M$, we can find some $y' \sqsupset x'$ in $N$ such that $N,y' \rightleftharpoons M,x$. From bisimilarity, this means there is some $y \sqsupset x$ in $M$ such that $M,x \rightleftharpoons M,y$. However since $M \vDash \varphi(w)$, this can only imply that there is an infinite $\prec$ path in $M$, a contradiction.
    \end{description}
    Now we will define a class of rooted models $\{ M_i, w_i\}_{i < \omega}$ such that $M_i \vDash \varphi(w_i)$ and if $N \subseteq M_i$ and $N \vDash \varphi(w_i)$ then there is some $j \leq i$ such that $N,w_i \rightleftharpoons M_j , w_j$. The models are defined as follows:
    \begin{enumerate}
        \item $M_0 = \la C , \sqsubset, \| \cdot \| \ra $, where $C= \{ a,b,c\}$, $\sqsubseteq = \{ \la x, y\ra : x \neq y\}$ and $\| p\| = \{ b,c\}$. We let $w_0=b$.
        \item $M_{i+1} = \left[ \ontop{M_i}{C} \right]$, whose root $w_i$ is the point $b$ of the bottom cluster $C$.
    \end{enumerate}
    The models $M_i$ are finite chains of the cluster $C$ and it is easy to show that they satisfy the above two conditions. Now, working towards a contradiction, suppose there is some formula $\xi$ of $\mathcal{L}_\mu$ equivalent to $\varphi(x)$ and let $\Sigma = Cl_{\diamonddot, sub^*,\neg}(\xi)$ and $n = \left| \Sigma \right|+2$. But then, not every cluster of $M_n$ can be $\Sigma$-final and by Theorem \ref{thm: in between model and its final part}, we can find some $N \subseteq M_n$ such that $w_n \vDash_N \xi$ but $N \not \vDash \varphi (w_n)$.
\end{proof}

\longversion{
In addition to that, we get that formulas of $\mathcal{L}_\mu$ are equivalent to formulae without nested fixed points i.e. formulae that are not equal to formulae of the form $\rho_0 x. \varphi(\rho_1 y. \psi)$ with $x$ occurring in $\psi$ and $\rho_i$ being $\nu$ or $\mu$. Hence our work has as a corollary the following result by \cite{alternation}:
\begin{theorem}
    The $\mu$-calculus collapses to its alternation-free fragment over {\sf wK4} frames.
\end{theorem}

In view of \cite{BBFMu}, these results lift to topological spaces, with the caveat that $\sf FOL$ does not make sense in this setting.

\begin{theorem}
    $\mathcal{L}_{\mu} \equiv \mathcal{L}_{\vardiamond^\infty}$ over the class of all topological spaces.
\end{theorem}

% \section{Upper bounds and the collapse of the alternation hierarchy}
}

We can obtain a rough estimate of $|\chi(\varphi)| \leq 2^{(14|\varphi|+1)2^{14|\varphi|+6}}$.
This upper bound also applies in the transitive setting, whereas it is more difficult to extract from the methods of~\cite{DO}.
This bound is reasonably close to the known lower bound, which is exponential~\cite{Iliev}.
Finding the optimal size of a translation remains an interesting open problem.

\begin{toappendix}

\begin{theorem}
    Let $\varphi \in \mathcal{L}_\mu$ and $|\varphi|=n$ be the total number of symbols that appear in $\varphi$. Then $|\chi(\varphi)|\leq 2^{(14n+1)2^{14n+6}}$.
    Hence for every formula $\varphi$ of $\mathcal L_\mu$ there is a formula of $\mathcal L_{\vardiamond^\infty}$ of size bounded by a double exponential function on $|\varphi|$ and equivalent to $\varphi$ over the class of weakly transitive frames as well as the class of all topological spaces. 
\end{theorem}

% \longversion{
\begin{proof} 
The number of propositional constants that show up in $\varphi$ is $|P|\leq \ceil{\dfrac{n}{2}} \leq n$; and since {\sf S4} only has finitely many induced modalities \cite{Chagrov1997ModalL}, we know in particular that the induced $\Sigma = Cl_{\diamonddot, sub^*,\neg}(\varphi)$ has cardinality $|\Sigma|\leq 14 n =: m$. Then the number of different propositional valuations a given world can have is $2^{|P|} \leq 2^n$ and thus the number of different non-bisimilar clusters of {\sf wK4} frames is $|Sat^0_0|\leq 3^{2^{|P|}} \leq 3^{2^n}$. The number of sets of the form $\{ \overline{\la i \ra }\psi : \psi \in \Sigma \wedge i< dpt_{\Sigma}(w) \wedge  w\vDash_M \overline{\la i \ra }\psi \} $ for $\Sigma$-semifinal rooted finite models $(M,w)$ of depth $k$ is  $|Sat^{*1}_k|\leq 2^m (k+1)$ and so $|Sat_k| \leq |Sat^*_k| \leq |Sat^{*0}_k||Sat^{*1}_k| \leq  |Sat^{0}_0||Sat^{*1}_k|\leq 3^{2^n} 2^m (k+1) $. The number of (semi-)witnessing chains of depth $k$ is then $|Chain_k| \leq |Chain^*_k| \leq \prod_{0 \leq i \leq k} |Sat^*_i| \leq 3^{(k+1)2^{n}} 2^{m(k+1)} (k+1)^{k+1} \leq 2^{2^{m+2}} $.

For the formulas used to define the modal $\la k \ra \psi$'s, we will for the interest of clarity use a slightly modified notation when calculating the upper bound of the number of symbols showing up. Define $|\la k \ra| := sup \{ |\la k \ra \psi| : \psi \in \Sigma \}  $ and similarly $|\alpha(k)| := sup \{ | \alpha(\mathcal{C})| : \mathcal{C} \in Chain^*_k\}$ and $|\delta(k)| := sup \{ | \delta(\mathcal{C})| : \mathcal{C} \in Chain^*_k\}$. Then for given $k$, the definitions give
\begin{itemize}
    \item $|\la k \ra | \leq 4|Chain_k| \cdot |\delta(k)| \leq 2^{2^{m+2}+2}|\delta (k)|$
    \item $|\alpha(k)| \leq 1+ 2^n \big( 2^{n+3} + 4 + \sum_{i<k} (|\la i \ra | \cdot 3m ) + |\delta(k-1)| \big) $
    \item  $\displaystyle \begin{aligned}[t]
     &|\delta(k)| \leq 3|Chain^*_k|\cdot|\alpha(k)| \leq 
     3\cdot  2^{2^{m+2}} |\alpha(k)| \\
     &\leq 
     3\cdot  2^{2^{m+2}} \cdot
     2^{n+1}\big( 2^{n+3}+ 2^{2^{m+2}+2}\cdot 
     \sum_{i<k} (|\delta(i)| \cdot 3m)
     + |\delta(k-1)|\big)     \\
     &\leq 2^{2^{m+2}} \cdot 2^{n+7}\big(2^{2^{m+3}}\cdot (|\delta(k-1)|+1)\big) \\
     &\leq 2^{k\cdot 2^{m+4}}
    \end{aligned}$
    
\end{itemize}
Then we can evaluate an upper bound of $|split(k)| \leq 6|Chain_{k+1}^*|  |\delta(k)| \leq 6\cdot 2^{2^{m+1}} \cdot 2^{k\cdot 2^{m+4}} \leq 2^{(k+1)\cdot 2^{m+4}} $ and finally $|\chi (\varphi)| \leq 
2^{n} \cdot |\bigcup_{k\leq m} Chain_k| \cdot |Sat^{*1}_m| \cdot (2|\la m+1 \ra| +  2^{n+3} + |split(m)| + 2|\delta(m)|) \leq
2^n \cdot (m+1)\cdot 2^{2^{m+2}} \cdot 2^m \cdot (m+1) \cdot (2^{2^{m+4}+2^{m+2}+3} + 2^{n+3 } + 2^{(m+1)\cdot 2^{m+4}} + 2^{m\cdot 2^{m+4}}) \leq 2^{(m+1)2^{m+6}} $.
   
\end{proof}
% }

\end{toappendix}

% \longversion{

%%%%%%%%%%%%%%%%%%%%%%%%%%%%%%%%%%%%%%%%%%%%%%%%%%%%%%%%%%%%%%%
% We use the definition of the alternation hierarchy as in \cite{alternation}.
% \begin{definition}
%     \begin{itemize}
%         \item $\Sigma_0 ^\mu = \Pi_0 ^\mu$ is the set of all $\mu$-formulas with no fixed-point operators;
%         \item $\Sigma_{n+1} ^\mu$ is the closure of $\Sigma_{n} ^\mu \cup \Pi_{n} ^\mu$ under propositional symbols, modal operators, $\mu x$ and the substitution: for a $\varphi(x), \varphi \in \Sigma_{n+1} ^\mu$ such that no free variable of $\psi$ is captured by $\varphi$, $\varphi(\psi) \in \Sigma_{n+1} ^\mu$;
%         \item $\Pi_{n+1} ^\mu$ is the closure of $\Sigma_{n} ^\mu \cup \Pi_{n} ^\mu$ under propositional symbols, modal operators, $\mu x$ and the analogous substitution;
%         \item $\Delta_n ^\mu = \Sigma_n ^\mu \cap \Pi_n ^\mu $.
%     \end{itemize}
%     The alternation hierarchy collapses iff 
% \end{definition}
%%%%%%%%%%%%%%%%%%%%%%%%%%%%%%%%%%%%%%%%%%%%%%%%%%%%%%%%%%%%%%%%%

\section{Conclusion}

We have shown that the topological $\mu$-calculus is equi-expressive to its tangled fragment, provided it's defined in a way that better captures its intended behaviour on arbitrary topological spaces while retaining its original value on metric spaces and other `nice' topological spaces.
Given the much more transparent syntax of tangle logic, this suggests that the latter is more suitable for applications in spatial KR than the full $\mu$-calculus.

% In future work we will present our proof that the $\mu$-calculus is not as expressive as $\sf MSO$ over the class of finite, weakly transitive frames, hence we do not obtain a modal characterization theorem in the sense of van Benthem.
This begs the question of whether the topological $\mu$-calculus, or its tangled fragment, can be enriched in a natural way to obtain the full expressive power of the bisimulation-invariant fragments of $\sf FOL$ or $\sf MSO$.
Perhaps something in the spirit of hybrid logics can bridge this gap, but at this point the question remains a challenging open problem.
\bibliographystyle{splncs04}
% \bibliographystyle{apalike}
\bibliography{biblio}
\end{document}